\documentclass[12pt,draft]{article} 
\usepackage{amsmath,amssymb,amsthm,amsfonts,bm}
\usepackage{enumerate,color}

\makeatletter
\def\@cite#1#2{[{{\bfseries #1}\if@tempswa , #2\fi}]}
\renewcommand{\section}{%
\@startsection{section}{1}{\z@}
{0.5truecm plus -1ex minus -.2ex}%
{1.0ex plus .2ex}{\bfseries\large}}
\def\@seccntformat#1{\csname the#1\endcsname.\ }
\makeatother

\numberwithin{equation}{section} 
\newtheorem{thm}{Theorem}[section]

\newtheorem{lem}[thm]{Lemma}

\theoremstyle{definition}

\newtheorem{remark}{Remark}[section]

\newcommand{\ep}{\varepsilon}
\newcommand{\pa}{\partial}
\newcommand{\Rn}{\mathbb{R}^n}
\newcommand{\uast}{u^{\ast}}
\newcommand{\vast}{v^{\ast}}
\newcommand{\wast}{w^{\ast}}

\newcommand{\ol}[1]{\overline{#1}}

\newcommand{\lp}[2]{\Vert{#2}\Vert_{L^{#1}(\Omega)}}

\newcommand{\cd}{(\cdot,t)}
\newcommand{\into}{\int_\Omega}

\begin{document}
\footnote[0]
    {2010{\it Mathematics Subject Classification}\/. 
    Primary: 35K51; Secondary: 92C17, 35B40.
    }
\footnote[0]
    {\it Key words and phrases\/: 
    chemotaxis; 
    Lotka--Volterra; 
    global existence; stabilization. 
    }
\begin{center}
    \Large{{\bf 
    Improvement of conditions for 
    asymptotic stability 
    in a two-species chemotaxis-competition model \\
    with signal-dependent sensitivity
    }}
\end{center}
\vspace{5pt}
\begin{center}
    Masaaki Mizukami 
   \footnote[0]{
    E-mail: 
    {\tt masaaki.mizukami.math@gmail.com} 
    }\\
    \vspace{12pt}
    Department of Mathematics, 
    Tokyo University of Science\\
    1-3, Kagurazaka, Shinjuku-ku, 
    Tokyo 162-8601, Japan\\
    \vspace{2pt}
\end{center}
\begin{center}    
    \small \today
\end{center}

\vspace{2pt}
\newenvironment{summary}
{\vspace{.5\baselineskip}\begin{list}{}{%
     \setlength{\baselineskip}{0.85\baselineskip}
     \setlength{\topsep}{0pt}
     \setlength{\leftmargin}{12mm}
     \setlength{\rightmargin}{12mm}
     \setlength{\listparindent}{0mm}
     \setlength{\itemindent}{\listparindent}
     \setlength{\parsep}{0pt}
     \item\relax}}{\end{list}\vspace{.5\baselineskip}}
\begin{summary}
{\footnotesize {\bf Abstract.}
This paper deals with the two-species 
chemotaxis-competition system 
\begin{equation*}
  \begin{cases}
    u_t=d_1\Delta u - \nabla \cdot (u \chi_1(w)\nabla w)
    +\mu_1 u(1-u-a_1 v) 
    & {\rm in} \ \Omega \times (0,\infty), \\
    v_t=d_2\Delta v - \nabla \cdot (v \chi_2(w)\nabla w)
    +\mu_2 v(1-a_2u-v) 
    & {\rm in} \ \Omega \times (0,\infty), \\
    w_t=d_3\Delta w + \alpha u + \beta v - \gamma w  
    & {\rm in} \ \Omega \times (0,\infty),
   \end{cases}
\end{equation*}
where $\Omega$ is a bounded domain in $\Rn$ 
with smooth boundary $\pa \Omega$, $n\ge 2$; 
$\chi_i$ are functions satisfying some conditions. 
About this problem, 
Bai--Winkler \cite{Bai-Winkler_2016} 
first obtained asymptotic stability in \eqref{cp} 
under some conditions 
in the case that $a_1,a_2\in (0,1)$. 
Recently, the conditions assumed in 
\cite{Bai-Winkler_2016} were improved (\cite{Mizukami_DCDSB}); 
however, there is a gap between 
the conditions assumed in 
\cite{Bai-Winkler_2016} and \cite{Mizukami_DCDSB}.
The purpose of this work is to 
improve the conditions assumed in 
the previous works 
for asymptotic behavior 
in the case that $a_1,a_2\in (0,1)$. 
}
\end{summary}
\vspace{10pt} 

\newpage
%
%

\section{Introduction}

%
%
This paper presents improvement of 
\cite{Bai-Winkler_2016,Mizukami_DCDSB}. 
In this paper we consider 
the two-species chemotaxis 
system with competitive kinetics 
\begin{equation}\label{cp}
  \begin{cases}
    u_t=d_1\Delta u - \nabla \cdot (u\chi_1(w) \nabla w)
    +\mu_1 u(1-u-a_1 v), 
    & x\in\Omega,\ t>0, 
\\[1mm]
    v_t=d_2\Delta v - \nabla \cdot (v\chi_2(w) \nabla w)
    +\mu_2 v(1-a_2u-v), 
    & x\in\Omega,\ t>0,  
\\[1mm]	
    w_t=d_3\Delta w + \alpha u + \beta v - \gamma w, 
    & x\in\Omega,\ t>0, 
\\[1mm]
    \nabla u\cdot \nu=\nabla v\cdot \nu = \nabla w\cdot \nu = 0, 
    & x\in\pa \Omega,\ t>0,
\\[1mm]
    u(x,0)=u_0(x),\ v(x,0)=v_0(x),\ w(x,0)=w_0(x)
    & x\in\Omega,
  \end{cases}
\end{equation}
where $\Omega$ is a bounded domain in $\Rn$ 
($n\ge 2$) 
with smooth boundary $\pa \Omega$ 
and $\nu$ 
is 
the 
outward normal vector to $\pa\Omega$; 
$d_1,d_2,d_3,\mu_1,\mu_2, 
a_1,a_2$ 
and $\alpha,\beta,\gamma$ 
are positive constants; $\chi_1,\chi_2,u_0,v_0,w_0$ are
assumed to be nonnegative functions.
The unknown functions $u(x,t)$ and $v(x,t)$  
represent the population densities of 
two species and 
$w(x,t)$ shows the concentration of the 
chemical substance 
at place $x$ and time $t$. 

The problem \eqref{cp}, 
which is proposed by Tello--Winkler \cite{Tello_Winkler_2012}, 
is a problem 
on account of the influence of chemotaxis, diffusion, 
and the Lotka--Volterra competitive kinetics, 
i.e., with coupling coefficients $a_1,a_2>0$ in 
\begin{align}\label{L-V}
  u_t = u(1-u-a_1 v), \quad v_t = v(1-a_2u-v). 
\end{align}
The mathematical difficulties of the problem \eqref{cp} 
are to deal with the chemotaxis term $\nabla \cdot (u\nabla w)$ 
and the competition term $u (1-u-a_1v)$. 
To overcome these difficulties, firstly, 
the parabolic-parabolic-elliptic problem 
(i.e., $w_t$ is replaced with $0$ in \eqref{cp}) 
was studied and 
some conditions for global existence and 
stabilization in \eqref{cp} were established 
(\cite{Black-Lankeit-Mizukami_01,
stinner_tello_winkler,Tello_Winkler_2012}). 
In the parabolic-parabolic-elliptic case 
global existence of classical solutions 
to \eqref{cp} and their asymptotic behavior 
were obtained in the case that $a_1,a_2\in (0,1)$ 
(\cite{Black-Lankeit-Mizukami_01,Tello_Winkler_2012})
and the case that $a_1 > 1 > a_2$ (\cite{stinner_tello_winkler}). 
Recently, 
these results which give 
global existence and stabilization in \eqref{cp} 
were improved in some cases (\cite{Mizukami_PPEpro}). 

On the other hand, in general, 
the fully parabolic problem \eqref{cp} is a more 
difficult problem than the 
parabolic-parabolic-elliptic case; 
because we cannot use the relation 
\[
\Delta w = \alpha u + \beta v - \gamma w. 
\]
About this problem, global existence and boundedness 
were shown in the 2-dimensional case 
(\cite{Bai-Winkler_2016}) and the $n$-dimensional 
case (\cite{Lin-Mu-Wang}). 
Moreover, in the case that $a_1,a_2\in (0,1)$, 
Bai--Winkler \cite{Bai-Winkler_2016} obtained 
asymptotic stability in \eqref{cp} 
under the conditions 
\begin{align}\label{condi;Bai-Winkler} 
  \mu_1>\frac{d_2\chi_1^2u^\ast}{\frac{4a_1\gamma (1-a_1a_2)d_1d_2d_3}{(a_1\alpha^2+a_2\beta^2-2a_1a_2\alpha\beta)}-\frac{d_1a_1\chi_2^2v^\ast}{4\mu_2a_2}}, 
\quad 
  \mu_2>\frac{\chi_2^2v^\ast (a_1\alpha^2+a_2\beta^2-2a_1a_2\alpha\beta)}{16d_2d_3a_2\gamma (1-a_1a_2)}. 
\end{align}
Recently, in \cite{Mizukami_DCDSB}, 
the conditions \eqref{condi;Bai-Winkler} were improved; 
asymptotic behavior of solutions holds 
when there exists $\delta_1>0$ satisfying 
$4\delta_1-a_1a_2(1+\delta_1)^2>0$, 
\begin{align}\label{condi;Mizukami}
  &\mu_1
  >
  \frac{M_1^2u^\ast(1+\delta_1)(\alpha_2^2a_1\delta_1
    +\beta_2^2a_2-\alpha_1\beta_1a_1a_2(1+\delta_1))}
  {4a_1d_1d_3\gamma(4\delta_1-a_1a_2(1+\delta_1)^2)},
\\\label{condi;Mizukami2}
    &\mu_2 >
    \frac{M_2^2v^\ast(1+\delta_1)
    (\alpha_2^2a_1\delta_1
    +\beta_2^2a_2-\alpha_1\beta_1a_1a_2(1+\delta_1))}
    {4a_2d_2d_3\gamma(4\delta_1-a_1a_2(1+\delta_1)^2)},
\end{align}
where $M_1,M_2>0$ are some constants satisfying 
$\chi_1(s) \le M_1$, $\chi_2(s)\le M_2$ 
for all $s\ge 0$. 
%
Here we note that 
the conditions 
\eqref{condi;Bai-Winkler} and \eqref{condi;Mizukami}--\eqref{condi;Mizukami2} 
can be rewritten as 
\begin{align}\label{rewrite;Bai-Winkler}
\left(\frac{\uast \chi_1^2}{4d_1d_3a_1\mu_1},
  \frac{\vast \chi_2^2}{4d_2d_3a_2\mu_2}
  \right) 
  &\in 
    \left\{
    (s,t)\in \mathbb{R}^2 \,|\, 
    s,t \ge 0,\ s + t < 
    f(1)
    \right\}
\end{align}
and 
\begin{align}\notag 
\left(\frac{\uast M_1^2}{4d_1d_3a_1\mu_1},
  \frac{\vast M_2^2}{4d_2d_3a_2\mu_2}
  \right) 
  &\in \bigcup_{q\in I}
  \left\{
    (s,t)\in \mathbb{R}^2 \,\Big|\, 
    0\le s< \frac{f(q)}{1+q},\ 
    0\le t < \frac{f(q)}{1+q}
    \right\}
    \\\label{rewrite;Mizukami} &
    = \left\{
    (s,t)\in \mathbb{R}^2 \,\Big|\, 
    0\le s< \frac{f(q_0)}{1+q_0},\ 
    0\le t < \frac{f(q_0)}{1+q_0}
    \right\},
\end{align}
respectively, 
where 
\begin{align}\label{def;Iandf}
  I:= 
   \{q>0\,|\, 
   4q - (1+q)^2 a_1a_2 > 0\}, 
\quad 
  f(q):= \frac{\gamma (4q - (1+q)^2a_1a_2)}
    {a_1\alpha^2 q + a_2\beta^2 -a_1a_2\alpha\beta (1+q)} 
\end{align} 
and $q_0\in I$ is a maximizer of $\frac{f(q)}{1+q}$. 
The regions derived from 
\eqref{rewrite;Bai-Winkler} 
and \eqref{rewrite;Mizukami} 
are described in Figure 1. 
\begin{figure}[h]
\begin{center}
{\unitlength 0.1in%
\begin{picture}(24.9000,19.7000)(2.0000,-23.2000)%
\put(4.6000,-22.2000){\makebox(0,0)[rt]{O}}%
\put(4.9000,-6.3000){\makebox(0,0)[rt]{$t$}}%
\put(26.9000,-22.2000){\makebox(0,0)[rt]{$s$}}%
%
\special{pn 8}%
\special{pa 410 2180}%
\special{pa 2690 2180}%
\special{fp}%
\special{sh 1}%
\special{pa 2690 2180}%
\special{pa 2623 2160}%
\special{pa 2637 2180}%
\special{pa 2623 2200}%
\special{pa 2690 2180}%
\special{fp}%
\special{pn 8}%
\special{pa 410 2180}%
\special{pa 2690 2180}%
\special{fp}%
%
\special{pn 13}%
\special{pa 540 820}%
\special{pa 2030 2180}%
\special{fp}%
%
\special{pn 4}%
\special{pa 1200 1440}%
\special{pa 540 2100}%
\special{fp}%
\special{pa 1230 1470}%
\special{pa 540 2160}%
\special{fp}%
\special{pa 1260 1500}%
\special{pa 580 2180}%
\special{fp}%
\special{pa 1300 1520}%
\special{pa 640 2180}%
\special{fp}%
\special{pa 1330 1550}%
\special{pa 700 2180}%
\special{fp}%
\special{pa 1360 1580}%
\special{pa 760 2180}%
\special{fp}%
\special{pa 1390 1610}%
\special{pa 820 2180}%
\special{fp}%
\special{pa 1420 1640}%
\special{pa 880 2180}%
\special{fp}%
\special{pa 1450 1670}%
\special{pa 940 2180}%
\special{fp}%
\special{pa 1480 1700}%
\special{pa 1000 2180}%
\special{fp}%
\special{pa 1520 1720}%
\special{pa 1060 2180}%
\special{fp}%
\special{pa 1550 1750}%
\special{pa 1120 2180}%
\special{fp}%
\special{pa 1580 1780}%
\special{pa 1180 2180}%
\special{fp}%
\special{pa 1610 1810}%
\special{pa 1240 2180}%
\special{fp}%
\special{pa 1640 1840}%
\special{pa 1300 2180}%
\special{fp}%
\special{pa 1670 1870}%
\special{pa 1360 2180}%
\special{fp}%
\special{pa 1700 1900}%
\special{pa 1420 2180}%
\special{fp}%
\special{pa 1740 1920}%
\special{pa 1480 2180}%
\special{fp}%
\special{pa 1770 1950}%
\special{pa 1540 2180}%
\special{fp}%
\special{pa 1800 1980}%
\special{pa 1600 2180}%
\special{fp}%
\special{pa 1830 2010}%
\special{pa 1660 2180}%
\special{fp}%
\special{pa 1860 2040}%
\special{pa 1720 2180}%
\special{fp}%
\special{pa 1890 2070}%
\special{pa 1780 2180}%
\special{fp}%
\special{pa 1920 2100}%
\special{pa 1840 2180}%
\special{fp}%
\special{pa 1950 2130}%
\special{pa 1900 2180}%
\special{fp}%
\special{pa 1990 2150}%
\special{pa 1960 2180}%
\special{fp}%
\special{pa 1170 1410}%
\special{pa 540 2040}%
\special{fp}%
\special{pa 1140 1380}%
\special{pa 540 1980}%
\special{fp}%
\special{pa 1110 1350}%
\special{pa 540 1920}%
\special{fp}%
\special{pa 1080 1320}%
\special{pa 540 1860}%
\special{fp}%
\special{pa 1040 1300}%
\special{pa 540 1800}%
\special{fp}%
\special{pa 1010 1270}%
\special{pa 540 1740}%
\special{fp}%
\special{pa 980 1240}%
\special{pa 540 1680}%
\special{fp}%
\special{pa 950 1210}%
\special{pa 540 1620}%
\special{fp}%
\special{pa 920 1180}%
\special{pa 540 1560}%
\special{fp}%
\special{pa 890 1150}%
\special{pa 540 1500}%
\special{fp}%
\special{pa 860 1120}%
\special{pa 540 1440}%
\special{fp}%
\special{pa 830 1090}%
\special{pa 540 1380}%
\special{fp}%
\special{pa 790 1070}%
\special{pa 540 1320}%
\special{fp}%
\special{pa 760 1040}%
\special{pa 540 1260}%
\special{fp}%
\special{pa 730 1010}%
\special{pa 540 1200}%
\special{fp}%
\special{pa 700 980}%
\special{pa 540 1140}%
\special{fp}%
\special{pa 670 950}%
\special{pa 540 1080}%
\special{fp}%
\special{pa 640 920}%
\special{pa 540 1020}%
\special{fp}%
\special{pa 610 890}%
\special{pa 540 960}%
\special{fp}%
\special{pa 570 870}%
\special{pa 540 900}%
\special{fp}%
%
\special{pn 4}%
\special{pa 1680 1660}%
\special{pa 560 1660}%
\special{fp}%
\special{pa 1680 1720}%
\special{pa 560 1720}%
\special{fp}%
\special{pa 1680 1780}%
\special{pa 560 1780}%
\special{fp}%
\special{pa 1680 1840}%
\special{pa 560 1840}%
\special{fp}%
\special{pa 1680 1900}%
\special{pa 560 1900}%
\special{fp}%
\special{pa 1680 1960}%
\special{pa 560 1960}%
\special{fp}%
\special{pa 1680 2020}%
\special{pa 560 2020}%
\special{fp}%
\special{pa 1680 2080}%
\special{pa 560 2080}%
\special{fp}%
\special{pa 1680 2140}%
\special{pa 560 2140}%
\special{fp}%
\special{pa 1680 1600}%
\special{pa 560 1600}%
\special{fp}%
\special{pa 1680 1540}%
\special{pa 560 1540}%
\special{fp}%
\special{pa 1680 1480}%
\special{pa 560 1480}%
\special{fp}%
\special{pa 1680 1420}%
\special{pa 560 1420}%
\special{fp}%
\special{pa 1680 1360}%
\special{pa 560 1360}%
\special{fp}%
\special{pa 1680 1300}%
\special{pa 560 1300}%
\special{fp}%
\special{pa 1680 1240}%
\special{pa 560 1240}%
\special{fp}%
\special{pa 1680 1180}%
\special{pa 560 1180}%
\special{fp}%
%
\special{pn 13}%
\special{pa 540 1130}%
\special{pa 1680 1130}%
\special{fp}%
%
\special{pn 13}%
\special{pa 1680 1130}%
\special{pa 1680 2180}%
\special{fp}%
\put(20.1000,-22.6000){\makebox(0,0)[lt]{$f(1)$}}%
\put(2.0000,-8.0000){\makebox(0,0)[lt]{$f(1)$}}%
\put(15.7000,-22.7000){\makebox(0,0)[lt]{$\frac{f(q_0)}{1+q_0}$}}%
\put(2.3000,-10.7000){\makebox(0,0)[lt]{$\frac{f(q_0)}{1+q_0}$}}%
\put(21.9000,-19.3000){\makebox(0,0)[lt]{Bai--Winkler [{\bf 1}]}}%
\put(21.4000,-13.1000){\makebox(0,0)[lt]{Mizukami [{\bf 6}]}}%
%
\special{pn 8}%
\special{pa 540 2260}%
\special{pa 540 650}%
\special{fp}%
\special{sh 1}%
\special{pa 540 650}%
\special{pa 520 717}%
\special{pa 540 703}%
\special{pa 560 717}%
\special{pa 540 650}%
\special{fp}%
%
\special{pn 13}%
\special{pa 2080 1380}%
\special{pa 1540 1380}%
\special{fp}%
\special{sh 1}%
\special{pa 1540 1380}%
\special{pa 1607 1400}%
\special{pa 1593 1380}%
\special{pa 1607 1360}%
\special{pa 1540 1380}%
\special{fp}%
%
\special{pn 13}%
\special{pa 2140 2000}%
\special{pa 1760 2000}%
\special{fp}%
\special{sh 1}%
\special{pa 1760 2000}%
\special{pa 1827 2020}%
\special{pa 1813 2000}%
\special{pa 1827 1980}%
\special{pa 1760 2000}%
\special{fp}%
\end{picture}}%
\\[4mm]
Figure 1.
\label{fig;previouswork}
\end{center}
\end{figure}

\noindent
These results \cite{Bai-Winkler_2016,Mizukami_DCDSB} 
were also concerned with asymptotic stability in \eqref{cp} 
in the case that $a_1 \ge 1 > a_2$. 
More related works can be found in 
\cite{Mizukami_AIMSmath,Mizukami_DCDSB,Mizukami-Yokota_01,N-T_SIAM,N-T_JDE,Zhang-Li_2015}; 
global existence and 
boundedness in \eqref{cp} 
with general sensitivity functions 
can be found in \cite{Mizukami_DCDSB,Zhang-Li_2015}; 
related works which treated 
the non-competition case are 
in \cite{Mizukami_AIMSmath,Mizukami-Yokota_01,
N-T_SIAM,N-T_JDE}.

In summary the conditions for asymptotic stability 
in \eqref{cp} are known; however, there is a gap 
between the conditions \eqref{condi;Bai-Winkler} 
and \eqref{condi;Mizukami}--\eqref{condi;Mizukami2} 
(for more details, see Figure 1). 
The purpose of this work is 
to improve the conditions assumed in 
\cite{Bai-Winkler_2016} and \cite{Mizukami_DCDSB} 
for asymptotic behavior 
in the case that $a_1,a_2\in (0,1)$. 
%
%
%
%
In order to attain this purpose 
we shall assume throughout this paper 
that there exists 
$M_1,M_2>0$ satisfying 
\begin{align}\label{condi;chi}
   &\chi_1(s) \le M_1, 
\quad 
   \chi_2(s) \le M_2
\quad 
  \mbox{for all}\ s>0, 
\\\label{condi;main}
  &  \left(\frac{\uast M_1^2}{4d_1d_3a_1\mu_1},
  \frac{\vast M_2^2}{4d_2d_3a_2\mu_2}
  \right) 
  \in \bigcup_{q\in I}
    \{(s,t)\in \mathbb{R}^2 \,|\, 
    s,t \ge 0,\ s + qt < f(q)\},
\end{align}
where the interval $I$ and the function $f$ 
are defined as \eqref{def;Iandf}. 
The region derived from the 
condition \eqref{condi;main} is described 
in Figure 2, 
and include the regions derived from 
\eqref{rewrite;Bai-Winkler} 
and \eqref{rewrite;Mizukami}.  
\begin{figure}[h]
\begin{center}
\input{test2}\\[4mm]
Figure 2.
\label{fig;thiswork}
\end{center}
\end{figure}

%
%
%
%
Now the main results 
read as follows. 
We suppose that the initial data 
$u_0, v_0, w_0$ satisfy
\begin{align}\label{ini} 
0\leq u_0\in C(\ol{\Omega})\setminus \{0\}, \quad
0\leq v_0\in C(\ol{\Omega})\setminus \{0\}, \quad 
0\leq w_0\in W^{1,q}(\Omega) \ (\exists \, q>n).
\end{align} 
The first theorem is concerned 
with asymptotic behavior 
in \eqref{cp} 
in the case $a_1,a_2\in (0,1)$. 
\begin{thm}\label{mainth} 
Let $d_1,d_2,d_3,\mu_1,\mu_2,
\alpha,\beta,\gamma>0$, 
$a_1,a_2\in (0,1)$ be constants, 
let $\chi_1,\chi_2$ be nonnegative functions 
and let $\Omega \subset \Rn$ $(n\ge 2)$ be a bounded domain 
with smooth boundary. 
Assume that there exists a unique global 
classical solution $(u,v,w)$ of \eqref{cp} 
satisfying 
\begin{align*}
 \|u\|_{C^{2+\theta,1+\frac{\theta}{2}}
 (\ol{\Omega}\times[t,t+1])} 
 + \|v\|_{C^{2+\theta,1+\frac{\theta}{2}}
 (\ol{\Omega}\times[t,t+1])} 
 + \|w\|_{C^{2+\theta,1+\frac{\theta}{2}}
 (\ol{\Omega}\times[t,t+1])}
 \le M
 \quad \mbox{for all}\ t\ge 1
\end{align*}
with some $\theta\in (0,1)$ and $M>0$. 
Then 
under the conditions 
\eqref{condi;chi}{\rm --}\eqref{ini}, 
the solution $(u,v,w)$ satisfies that there exist $C>0$ 
and $\ell>0$ such that 
\begin{align*}
 \|u\cd-u^\ast\|_{L^{\infty}(\Omega)} 
 +\|v\cd-v^\ast\|_{L^{\infty}(\Omega)} 
 +\|w\cd-w^\ast\|_{L^{\infty}(\Omega)} 
 \leq Ce^{-\ell t} 
 \quad 
  \mbox{for all}\ t>0, 
\end{align*} 
where 
\[
  u^\ast:=\frac{1-a_1}{1-a_1a_2},
\quad 
  v^\ast:=\frac{1-a_2}{1-a_1a_2}, 
\quad 
  w^\ast := \frac{\alpha \uast +\beta \vast}{\gamma}.
\]
\end{thm}
%
\begin{remark}
The condition \eqref{condi;main} improves 
the conditions assumed in 
\cite{Bai-Winkler_2016} and \cite{Mizukami_DCDSB} 
(for more details, see Section \ref{discussion}). 
Moreover, from the careful calculations we have that 
\begin{align*}
  &\bigcup_{q\in I}
    \{(s,t)\in \mathbb{R}^2 \,|\, 
    s,t \ge 0,\ s + qt < f(q)\} 
    \\
    &= \left\{(s,t)\in \mathbb{R}^2 \,\Big|\, 
    s\ge 0,\ 0\le t < \frac{a_2\gamma}{\alpha(a_2\beta-\alpha)_+},\
    h_1(s,t)>0,\ h_2^+(s,t)>0,\ h_2^-(s,t)<0
    \right\},
\end{align*}
where 
\begin{align*}
h_1(s,t):= \,&a_1^2\alpha^2(\alpha-a_2\beta)^2s^2 
           + a_2^2\beta^2 (\beta - a_1\alpha)^2 t^2 
           -2a_1a_2\alpha\beta(\alpha-a_2\beta)(\beta-a_1\alpha)st 
           \\ &
           - 4\gamma (2a_1\alpha^2 
           - 2a_1a_2\alpha\beta 
           + a_1a_2^2\beta^2
           - a_1^2a_2\alpha^2)s
           \\ &
           - 4\gamma (2a_2\beta^2\gamma 
           - 2a_1a_2\alpha\beta
           + a_1^2a_2\alpha^2
           - a_1a_2^2\beta^2)t
           + 16\gamma^2 (1-a_1a_2)
\end{align*}
and
\begin{align*}
h_2^{\pm}(s,t):=\, &a_1\alpha(\alpha-a_2\beta)s + 
          \left(
          \frac{\alpha(4-2a_1a_2 \pm 4\sqrt{1-a_1a_2})
          (\alpha-a_2\beta)
          }{a_2}
          +a_2\beta(\beta-a_1\alpha)
          \right) t
          \\ &
          \pm 4\gamma \sqrt{1-a_1a_2}.
\end{align*}
\end{remark}
\begin{remark}
We note from $q\in I$ that 
\begin{align*}
 a_1\alpha^2 q + a_2\beta^2 -a_1a_2\alpha\beta (1+q)>0
\end{align*}
holds. Indeed, 
the discriminant of 
$a_1q \alpha^2  -a_1a_2\beta (1+q)\alpha 
+ a_2\beta^2 $ 
is negative: 
\begin{align*}
 D_\alpha := 
 - a_1a_2\beta^2 (4q - (1+q)^2a_1a_2) < 0. 
\end{align*}
\end{remark}

Then a combination of results concerned with 
global existence and boundedness in \eqref{cp} 
(\cite{Bai-Winkler_2016,Lin-Mu-Wang,Mizukami_DCDSB}) 
and Theorem \ref{mainth} 
implies the following results. 

\begin{thm}\label{mainth2}
Let $d_1,d_2,d_3> 0$, $\mu_1,\mu_2 > 0$, 
$a_1,a_2\in (0,1)$, $\alpha,\beta,\gamma>0$ and 
let $\Omega \subset \Rn$ $(n\ge 2)$ be 
a bounded domain with smooth boundary. 
Assume that $\chi_1,\chi_2>0$ are constants and 
one of the following two properties is satisfied\/{\rm :} 
\begin{itemize}
\item[{\rm (i)}] 
$n=2$, 
\item[{\rm (ii)}]
$\Omega$ is a convex domain and 
$\mu_1>\frac{n\chi_1}{4}$, $\mu_2>\frac{n\chi_2}{4}$, 
$\mu_1+\frac{a_1\mu_1}{2}+\frac{a_2\mu_2\chi_1}{2\chi_2}>\frac{n\chi_1}{2}$ and 
$\mu_2+\frac{a_2\mu_2}{2}+\frac{a_1\mu_1\chi_2}{\chi_1}> \frac{n\chi_2}{2}$ 
hold.
\end{itemize}
Then 
under the conditions 
\eqref{condi;chi}{\rm --}\eqref{ini}, 
the same conclusion as in Theorem \ref{mainth} 
holds. 
\end{thm}

\begin{thm}\label{mainth3}
Let $d_1,d_2,d_3> 0$, $\mu_1,\mu_2 > 0$, 
$a_1,a_2\in (0,1)$, 
$\alpha,\beta,\gamma>0$ and 
let $\Omega \subset \Rn$ $(n\ge 2)$ be a bounded domain 
with smooth boundary. 
Assume that the functions $\chi_1,\chi_2$ 
satisfy the following conditions\/{\rm :} 
\begin{align*} 
  &\chi_i\in C^{1+\eta}([0,\infty))\cap 
  L^{1}(0,\infty), 
\\
  &2d_id_3\chi_i'(w)+\left( (d_3-d_i)p+
    \sqrt{(d_3-d_i)^2p^2+4d_id_3p}\right)
    [\chi_i(w)]^2\leq 0,
\\
  &w\chi_i(w) \le C_\chi \quad 
  \mbox{for all}\ w\ge 0\ \mbox{and}\ i=1,2,3
\end{align*}
with some $\eta\in (0,1)$, $p>n$ and $C_\chi>0$. 
Then under the conditions 
\eqref{condi;chi}{\rm --}\eqref{ini}, 
the same conclusion as in 
Theorem \ref{mainth} holds. 
\end{thm}
\begin{remark}
We note from the global-in-time lower estimate for 
$w$ (\cite{Fujie_2015}) that 
the same arguments as in the proof of 
Theorem \ref{mainth} 
can also be applied to 
the case that $\chi_i(w)=\frac{K_i}{w}$. 
\end{remark}

The strategy of the proof of Theorem \ref{mainth} 
is to modify the methods in 
\cite{Bai-Winkler_2016,Mizukami_DCDSB}. 
One of the keys for the proof of 
Theorem \ref{mainth} is to 
derive the following energy estimate: 
\begin{align}\label{strategy;stab}
 \frac{d}{dt}E(t) \le 
 -\ep \into 
 \left[
   (u\cd-\uast)^2 + (v\cd-\vast)^2 + (w\cd-\wast)^2
 \right]
 \quad \mbox{for all}\ t>0
\end{align}
with some positive function $E$ and 
some constant $\ep>0$. 
Thanks to \eqref{strategy;stab}, 
we can obtain Theorem \ref{mainth}. 
The key for the improvement is to 
provide the best estimate for the terms 
\begin{align*}
  \into \frac{\chi_1(w)}{u}\nabla u\cdot \nabla w
\quad \mbox{and}\quad 
  \into \frac{\chi_2(w)}{v}\nabla v\cdot \nabla w. 
\end{align*}

%
%

This paper is organized as follows. 
In Section 2 
we prove 
asymptotic stability in the case that $a_1,a_2\in (0,1)$
(Theorems \ref{mainth}, \ref{mainth2} and \ref{mainth3}). 
Section 3 is devoted to discussions; 
we confirm that the condition \eqref{condi;main} 
improves the conditions assumed 
in \cite{Bai-Winkler_2016} and \cite{Mizukami_DCDSB}. 

%
%

\section{Proof of Theorem \ref{mainth}}

In this section we prove stabilization in \eqref{cp} 
in the case that $a_1,a_2\in (0,1)$. 
Here we assume that 
there exists a unique global 
classical solution $(u,v,w)$ of \eqref{cp} 
satisfying 
\begin{align*}
 \|u\|_{C^{2+\theta,1+\frac{\theta}{2}}
 (\ol{\Omega}\times[t,t+1])} 
 + \|v\|_{C^{2+\theta,1+\frac{\theta}{2}}
 (\ol{\Omega}\times[t,t+1])} 
 + \|w\|_{C^{2+\theta,1+\frac{\theta}{2}}
 (\ol{\Omega}\times[t,t+1])}
 \le M
 \quad \mbox{for all}\ t\ge 1
\end{align*}
with some $\theta \in (0,1)$ and $M>0$. 
%
%
%
%
We first provide the following lemma which 
will be used later. 

\begin{lem}\label{lem;quadra}
Let $a,b,c,d,e,f\in \mathbb{R}$. 
Suppose that 
\begin{align}\label{assumption;lemma1}
  a>0,
\quad 
  ad-\frac{b^2}{4}>0, 
\quad
  adf + \frac{bce}{4} - \frac{cd^2}{4} 
  - \frac{b^2 f}{4} -\frac{ae^2}{4}>0.
\end{align}
Then there exists $\ep>0$ such that 
\begin{align*}
ay_1^2 +by_1y_2 + cy_1y_3 + dy_2^2 + ey_2y_3 + f y_3^2 
\ge \ep \left(y_1^2+ y_2^2+y_3^2\right)
\end{align*}
holds for all $y_1,y_2,y_3\in \mathbb{R}$. 
\end{lem}
\begin{proof}
In order to prove this lemma we shall see that 
there is $\ep>0$ such that 
\begin{align}\label{parpose;lemma1}
{}^tXA_\ep X = 
(a-\ep)y_1^2 +by_1y_2 + cy_1y_3 
+ (d-\ep)y_2^2 + ey_2y_3 + (f-\ep) y_3^2 
\ge 0  
\end{align}
holds for all $y_1,y_2,y_3\in \mathbb {R}$, 
where 
\begin{align*}
 X:= 
   \begin{pmatrix}
   y_1 \\
   y_2 \\
   \ y_3\ 
   \end{pmatrix}, 
\quad
 A_\ep := 
 \begin{pmatrix}
   \ a-\ep \ & \frac b2 & \frac c2 \\
   \frac b2  & \ d-\ep\ & \frac e2 \\
   \frac c2  & \frac e2 & \ f-\ep \
 \end{pmatrix}.
\end{align*}
To confirm that 
there is $\ep>0$ such that 
\eqref{parpose;lemma1} holds for all 
$y_1,y_2,y_3\in \mathbb{R}$ we 
put 
\begin{align*}
  g_1(\ep) := a-\ep, 
\quad
  g_2(\ep) := 
  \begin{vmatrix}
   \ a-\ep\ & \frac b2 \\
   \frac b2 & \ d-\ep\
  \end{vmatrix},
\quad
  g_3(\ep) := |A_\ep|
\end{align*}
and shall show the existence of $\ep_1>0$ satisfying 
$g_i(\ep_1)>0$ for $i=1,2,3$. 
Now thanks to \eqref{assumption;lemma1}, 
we can see that 
\begin{align*}
  g_1(0) = a>0, 
\quad  
  g_2(0) =  
   \begin{vmatrix}
     \ a \    & \frac b2 \\
     \frac b2 & \ d\ 
   \end{vmatrix}
  = ad -\frac{b^2}{4}>0
\end{align*}
and
\begin{align*}
 g_3(0) =  
 \left|
\begin{array}{ccc}
   \ a \    & \frac b2 & \frac c2 \\
   \frac b2 & \ d \    & \frac e2 \\
   \frac c2 & \frac e2 & \ f \
\end{array}
  \right|
  = adf + \frac{bce}{4} - \frac{cd^2}{4} 
  - \frac{b^2 f}{4} -\frac{ae^2}{4}>0.
\end{align*}
Thus a combination of the above inequalities and 
the continuity argument yields that 
there is $\ep_1>0$ such that 
$g_i (\ep_1)>0$ hold for $i=1,2,3$. 
Therefore aided by the Sylvester criterion, 
we have \eqref{parpose;lemma1} 
for all $y_1,y_2,y_3\in \mathbb{R}$, 
which means the end of the proof. 
\end{proof}
%
%
%
%
Then we will prove the following energy estimate 
which leads to asymptotic behavior of solutions 
to \eqref{cp}. 
The proof is mainly based on 
the methods in \cite{Bai-Winkler_2016,Mizukami_DCDSB}. 

\begin{lem}\label{case1ena}
 Let $a_1,a_2\in(0,1)$ and 
 let $(u,v,w)$ be a solution to \eqref{cp}. 
 Then under the conditions 
 \eqref{condi;chi}{\rm --}\eqref{ini}, 
 there exist 
 a nonnegative function 
 $E:(0,\infty)\to \mathbb{R}$ and a constant $\ep>0$ 
 such that 
\begin{align}\label{case1imineq}
  \frac{d}{dt}E(t)\leq 
  -\ep 
  \left(
  \into (u-\uast)^2 
  + \into (v-\vast)^2 
  + \into (w-\wast)^2
  + \into |\nabla w|^2
  \right)
\end{align}
holds for all $t>0$. 
\end{lem}
\begin{proof}
In light of \eqref{condi;main} 
we can take $q>0$ such that 
\begin{align*}
  &4q-(1+q)^2 a_1a_2 > 0
\end{align*}
and
\begin{align*}
  \frac{\uast M_1^2}{4d_1d_3a_1\mu_1} 
  + \frac{\vast qM_2^2}{4d_2d_3a_2\mu_2}
  < \frac{\gamma (4q - (1+q)^2a_1a_2)}
    {a_1\alpha^2 q + a_2\beta^2 -a_1a_2\alpha\beta (1+q)}
\end{align*}
hold and $\delta>0$ satisfying 
\begin{align*}
  \frac{\uast a_2\mu_2M_1^2}{4d_1} 
  + \frac{\vast qa_1\mu_1M_2^2}{4d_2}
  < 
  \delta 
  < \frac{\gamma a_1\mu_1a_2\mu_2 (4q - (1+q)^2a_1a_2)}
    {a_1\alpha^2 q + a_2\beta^2  -a_1a_2\alpha\beta (1+q)}.
\end{align*}
For all $t>0$ we denote by $A(t)$, $B(t)$, $C(t)$ 
the functions defined as 
\begin{align*}
  &A(t):=\int_\Omega \left(
  u\cd-u^\ast-u^\ast \log \frac{u\cd}{u^\ast}
  \right), \quad
\\
  &B(t):=\int_\Omega \left(
   v\cd-v^\ast-v^\ast \log \frac{v\cd}{v^\ast}
   \right), \quad
\\
   & C(t):= \frac{1}{2}\int_\Omega \left(
   w\cd-w^\ast
   \right)^2,
 \end{align*}
and shall confirm that the function 
$E(t)$ defined as 
\begin{align*}
  E(t) 
  := a_2\mu_2 A(t)
  + qa_1\mu_1 B(t)
  + \delta C(t)
\end{align*}
satisfies \eqref{case1imineq} with some $\ep>0$. 
Firstly 
the Taylor formula 
enables us to see that $E(t)$ is a 
nonnegative function for $t>0$ 
(for more details, see \cite[Lemma 3.2]{Bai-Winkler_2016}). 
From the straightforward calculations  we infer 
\begin{align*}
  \frac{d}{dt}A(t)&=
  -\mu_1\int_\Omega (u-u^\ast)^2
  -a_1\mu_1\int_\Omega (u-u^\ast)(v-v^\ast)
  -d_1u^\ast\int_\Omega\frac{|\nabla u|^2}{u^2}
\\
  &\quad\,
  +u^\ast\int_\Omega\frac{\chi_1(w)}{u}
  \nabla u\cdot\nabla w, 
\\
  \frac{d}{dt}B(t)&=
  -\mu_2\int_\Omega (v-v^\ast)^2
  -a_2\mu_2\int_\Omega (u-u^\ast)(v-v^\ast)
  -d_2v^\ast\int_\Omega\frac{|\nabla v|^2}{v^2}
\\
  &\quad\,
  +v^\ast\int_\Omega\frac{\chi_2(w)}{v}
  \nabla v\cdot\nabla w
  , \\
  \frac{d}{dt}C(t)&=
  \alpha \int_\Omega  (u-u^\ast)(w-w^\ast)
  + \beta \int_\Omega  (v-v^\ast)(w-w^\ast)
  - \gamma \int_\Omega  (w-w^\ast)^2
\\
  &\quad\,-d_3\int_\Omega |\nabla w|^2. 
\end{align*}
Hence we have 
\begin{align*}
  \frac{d}{dt}E(t)= I_1(t)+I_2(t),
\end{align*}
where
\begin{align*}
  I_1(t)&:= -a_2\mu_1\mu_2 \int_\Omega (u-u^\ast)^2 
  -a_1a_2\mu_1\mu_2(1+q)
  \int_\Omega (u-u^\ast)(v-v^\ast)
\\
  & \quad\, \
  +\delta \alpha \int_\Omega (u-u^\ast)(w-w^\ast)
  -a_1\mu_1\mu_2 q \int_\Omega (v-v^\ast)^2
\\ &\quad\,\
  +\delta \beta \int_\Omega (v-v^\ast)(w-w^\ast) 
  -\delta \gamma \int_\Omega (w-w^\ast)^2
\end{align*}
and
\begin{align*}
I_2(t)&:=
  -d_1a_2\mu_2u^\ast 
  \int_\Omega\frac{|\nabla u|^2}{u^2}
  +a_2\mu_2u^\ast \int_\Omega\frac{\chi_1(w)}{u}
  \nabla u\cdot\nabla w
  -d_2a_1\mu_1 v^\ast q 
  \int_\Omega\frac{|\nabla v|^2}{v^2}
\\
  &\quad\,\   
  + a_1\mu_1v^\ast q \int_\Omega\frac{\chi_2(w)}{v}
  \nabla v\cdot\nabla w
  -d_3\delta \int_\Omega |\nabla w|^2.
\end{align*}
In order to confirm that there is $\ep_1>0$ such that 
\begin{align}\label{I3ineq}
  I_1(t)\leq -\ep_1\left(
  \int_\Omega (u\cd-u^\ast)^2
  +\int_\Omega (v\cd-v^\ast)^2
  +\int_\Omega (w\cd-w^\ast)^2
  \right)
\end{align}
for all $t>0$ 
we will see that the assumption of Lemma \ref{lem;quadra} 
is satisfied with 
\[
  a = a_2\mu_1\mu_2, 
\quad  
  b = a_1a_2\mu_1\mu_2 (1+q), 
\quad  
  c = -\delta \alpha, 
\quad 
  d = a_1\mu_1\mu_2 q, 
\quad  
  e = -\delta \beta, 
\quad 
  f = \delta \gamma
\]
and 
\[
  y_1 = u(x,t) - \uast,
\quad
  y_2 = v(x,t) - \vast,
\quad
  y_3 = w(x,t) - \wast.
\]
From the definitions of $q,\delta>0$ we can see that 
\begin{align*}
  a = a_2\mu_1\mu_2>0,
\quad
  ad-\frac{b^2}{4}
  = \frac{a_1a_2\mu_1^2\mu_2^2}{4} 
  (4q - (1+q)^2a_1a_2)
  >0 
\end{align*}
and 
\begin{align*}
  &adf + \frac{bce}{4} - \frac{cd^2}{4} 
  - \frac{b^2 f}{4} -\frac{ae^2}{4}
\\
  &= \frac{\mu_1\mu_2 \delta}{4} 
  \big(
  a_1a_2\mu_1\mu_2\gamma (4q - (1+q)^2 a_1a_2) 
  - (a_1\alpha^2 q + a_2\beta^2 - (1+q)a_1a_2\alpha\beta)\delta
  \big)>0. 
\end{align*}
Hence thanks to Lemma \ref{lem;quadra}, 
there exists $\ep_1>0$ such that 
\eqref{I3ineq} holds. 
We next verify that there is $\ep_2 >0$ such that 
\begin{align}\label{I4ineq}
  I_2(t)\leq -\ep_2 \int_\Omega |\nabla w\cd|^2 
  \quad \mbox{for all}\ t>0. 
\end{align}
By virtue of the Young inequality, 
we infer from \eqref{condi;chi} that 
\begin{align*}
  a_2\mu_2u^\ast
  \int_\Omega \frac{\chi_1(w)}{u}
  \nabla u\cdot \nabla w
  &\leq 
  d_1a_2\mu_2 \uast 
  \int_\Omega \frac{|\nabla u|^2}{u^2}
  + \frac{a_2\mu_2 \uast M_1^2}{4d_1}
  \int_\Omega |\nabla w|^2
\end{align*}
and
\begin{align*}
  a_1\mu_1 q v^\ast\int_\Omega 
  \frac{\chi_2(w)}{v}\nabla v\cdot \nabla w
  &\leq 
  d_2a_1\mu_1q\vast 
  \int_\Omega \frac{|\nabla v|^2}{v^2}
  + 
  \frac{a_1\mu_1 \vast q M_2^2}{4d_2}
  \int_\Omega |\nabla w|^2, 
\end{align*}
which implies that 
\begin{align}\label{soon;I2}
  I_2(t) \le 
  -\left(
  d_3\delta - \frac{a_2\mu_2 \uast M_1^2}{4d_1} 
  - \frac{a_1\mu_1 \vast q M_2^2}{4d_2}
  \right) 
  \into |\nabla w\cd|^2
  \quad 
  \mbox{for all}\ t>0. 
\end{align}
Therefore plugging the 
definition of $\delta>0$ into \eqref{soon;I2} 
leads to the existence of $\ep_2>0$ such that 
\eqref{I4ineq} holds, which concludes the proof of this lemma.  
\end{proof}

%
%
%
%

Then we have the following desired estimate. 

\begin{lem}\label{lem;decay;Linfty}
Let $a_1,a_2\in (0,1)$ and assume that 
\eqref{condi;chi}{\rm --}\eqref{ini} 
are satisfied. 
Then there exist $C > 0$ and $\ell > 0$ such that 
\begin{align*}
\lp{\infty}{u\cd-\uast}
+\lp{\infty}{v\cd-\vast}
+\lp{\infty}{w\cd-\wast}
\le Ce^{-\ell t}
\quad \mbox{for all}\ t>0.
\end{align*}
\end{lem}
\begin{proof}
The same arguments as in the proofs of 
\cite[Theorems 3.3, 3.6 and 3.7]{Bai-Winkler_2016} 
enable us to obtain this lemma. 
\end{proof}

\begin{proof}[{\rm \bf Proof of Theorem \ref{mainth}}]
Lemma \ref{lem;decay;Linfty} immediately leads to 
Theorem \ref{mainth}. 
\end{proof}
\begin{proof}[{\rm \bf Proof of Theorems \ref{mainth2} and \ref{mainth3}}] 
A combination of results concerned with 
global existence and boundedness in \eqref{cp} 
(\cite{Bai-Winkler_2016,Lin-Mu-Wang,Mizukami_DCDSB}) 
and the standard parabolic regularity 
argument, along with Theorem \ref{mainth} 
implies Theorems \ref{mainth2} and \ref{mainth3}. 
\end{proof}
%
%

\section{Discussions}\label{discussion}
In this section we shall confirm that the condition 
\eqref{condi;main}
improves the conditions assumed in the previous works 
aided by the view of \eqref{rewrite;Bai-Winkler} 
and \eqref{rewrite;Mizukami}.  
Here since  
\begin{align*}
  \left\{
  (s,t)\in \mathbb{R}^2 \,\Big|\, 
  0\le s< \frac{f(q)}{1+q},\ 
  0\le t < \frac{f(q)}{1+q}
  \right\}\subset 
  \{(s,t)\in \mathbb{R}^2 \,|\, 
  s,t \ge 0,\ s + qt < f(q)\}
\end{align*}
holds for every $q\in I$, 
we can see that the condition 
\eqref{condi;main} improves the conditions assumed in 
\cite{Mizukami_DCDSB}. 
In order to accomplish the purpose of this section, 
noting that 
\begin{align*}
  &\{(s,t)\in \mathbb{R}^2 \,|\, 
  s,t \ge 0,\ s + qt < f(q)\}
\\
  &= \bigcup_{k \in [0,1]}
  \left\{
  (s,t)\in \mathbb{R}^2 \,\Big|\, 
    0 \le s < kf(q),\ 
    0\le t < (1-k)\frac{f(q)}{q} 
  \right\}
    =: R_{q},
\end{align*}
we will confirm that  
\begin{align}\label{purpose;section3}
  R_{1} \subsetneq \bigcup_{q\in I}R_{q}.
\end{align}
To see \eqref{purpose;section3} we shall see that 
$q=1$ is not the best choice, i.e., 
$q= 1$ is not a maximizer of the functions 
$f(q)$ and 
\begin{align*}
  g(q):= \frac{f(q)}{q}
  \quad \mbox{for}\ 
  q\in I. 
\end{align*}
From the straightforward calculations we infer that 
\begin{align}\notag
  &\frac{d f}{d q}(1) = 
  \frac{4\gamma (1-a_1a_2)(\beta- a_1\alpha)}
  {(a_1\alpha^2 + a_2\beta^2 -2a_1a_2\alpha\beta)^2},
\\\label{der;gh}
  &\frac{d g}{d q}(1) = 
  \frac{4\gamma (1-a_1a_2)(a_2\beta - \alpha)}
  {(a_1\alpha^2 + a_2\beta^2 -2a_1a_2\alpha\beta)^2}. 
\end{align}
Now we will divide into two cases and show that 
\eqref{purpose;section3} holds for each cases. 
\vspace*{2mm}
%
%
%
%

\noindent
\underline{\bf Case 1: 
$a_1\alpha \neq \beta$ and $a_2\beta \neq \alpha$.} \ 
In this case, by virtue of \eqref{der;gh}, we can see that 
\begin{align*}
  \frac{d f}{d q}(1) \neq 0 
\quad \mbox{and} \quad 
  \frac{d g}{d q}(1) \neq 0, 
\end{align*}
which means that $q=1$ is not a maximizer of 
$f$ and $g$. 
Thus one of the following two properties holds: 
\begin{align*}
&\mbox{\rm (Case 1-1)}\quad \exists\, q_1 \in I;\quad 
  f(q_1) > f(1), \quad g(q_1) > g(1),
\\
&\mbox{\rm (Case 1-2)}\quad \forall\, q \in I;\quad 
  f(q) > f(1),\quad g(q) \le g(1) \quad\mbox{or}
\\ 
  &\hspace{41mm} 
  f(q) \le f(1),\quad g(q) > h(1).
\end{align*}
Therefore we obtain that \eqref{purpose;section3} holds in this case. 
\vspace*{2mm}
%
%
%
%

\noindent
\underline{\bf Case 2: 
$a_1\alpha = \beta$ or $a_2\beta = \alpha$.} \
We first deal with the case that 
$a_1\alpha = \beta$. 
In this case, in view of the fact that 
\begin{align*}
f(q) = - \frac{\gamma a_1a_2}{\alpha (1-a_1a_2)}\cdot q
+ \frac{2\gamma (2-a_1a_2)}{\alpha (1-a_1a_2)} 
- \frac{\gamma a_1a_2}{\alpha (1-a_1a_2)}\cdot \frac{1}{q},
\end{align*}
$q=1$ is a maximizer of $f$. 
On the other hand, thanks to \eqref{der;gh} 
together with the fact that 
$\alpha - a_2 \beta = \alpha (1-a_1a_2)>0$, 
we can see that 
$q=1$ is not a maximizer of $g$; 
there is $q_2\in I$ such that 
\begin{align*}
g(q_2) > g(1). 
\end{align*}
Similarly, in the case that $a_2\beta = \alpha$, 
we infer that $q=1$ is a maximizer of $g$ and 
is not a maximizer of $f$; there exists $q_3\in I$ 
such that 
\begin{align*}
f(q_3) > f(1). 
\end{align*}
Hence we derive that \eqref{purpose;section3} holds also in this case. 

According to the above two cases, we can obtain that 
\eqref{purpose;section3} holds, 
which means the condition 
\eqref{condi;main} 
improves the conditions assumed in \cite{Bai-Winkler_2016}. 
Therefore we attain the purpose of this paper. 
%


\newpage


\end{document}